\documentclass[11pt]{amsart}

\usepackage{amsmath,amsthm,amsfonts,amssymb,bm,graphicx }

\usepackage
{hyperref}
\hypersetup{colorlinks=true,citecolor=blue,linkcolor=blue,urlcolor=blue}

\theoremstyle{plain}
\newtheorem{theorem}{Theorem}
\newtheorem{lemma}{Lemma}

\newtheorem{prop}[lemma]{Proposition}
\numberwithin{equation}{section}

\newcommand{\pr}{^\prime}
\newcommand{\prd}{^{\prime 2}}

\makeatletter
\DeclareRobustCommand\widecheck[1]{{\mathpalette\@widecheck{#1}}}
\def\@widecheck#1#2{%
    \setbox\z@\hbox{\m@th$#1#2$}%
    \setbox\tw@\hbox{\m@th$#1%
       \widehat{%
          \vrule\@width\z@\@height\ht\z@
          \vrule\@height\z@\@width\wd\z@}$}%
    \dp\tw@-\ht\z@
    \@tempdima\ht\z@ \advance\@tempdima2\ht\tw@ \divide\@tempdima\thr@@
    \setbox\tw@\hbox{%
       \raise\@tempdima\hbox{\scalebox{1}[-1]{\lower\@tempdima\box
\tw@}}}%
    {\ooalign{\box\tw@ \cr \box\z@}}}
\makeatother

\begin{document}

\author{Valentin Blomer}
\author{V\' \i t\v ezslav Kala}
\address{Mathematisches Institut, Bunsenstr.~3-5, D-37073 G\"ottingen, Germany} \email{blomer@uni-math.gwdg.de} \email{vita.kala@gmail.com}

\title{Number fields without $n$-ary universal quadratic forms}
 
\thanks{First    author  supported by the Volkswagen Foundation and a Starting Grant of the European Research Council. Second author supported by a Starting Grant of the European Research Council.}


\begin{abstract} Given any positive integer $M$, we show that there are in\-finitely many real quadratic fields that do not admit universal quadratic forms with even cross coefficients in $M$ variables. 
\end{abstract}


\setcounter{tocdepth}{2}  \maketitle 

\section{Introduction} 

A famous theorem of Lagrange states that every positive integer is a sum of four squares, and one can rephrase this by saying that the   quadratic form $x_1^2 + x_2^2 + x_3^2 + x_4^2$ is universal, where here and in the following we call a positive form universal if it represents all positive integers. 

It is a classical fact that there exists no integral, positive,  ternary, universal quadratic forms (see e.g.\ \cite{Ro}), since every such form   misses a whole residue class (such as $x_1^2 + x_2^2+x_3^2 $ misses all integers $\equiv 7$ (mod 8)). Therefore the smallest number of variables, for which positive universal quadratic forms exist, is four. In fact, universal forms over $\mathbb Z$ can be characterized very easily: if they represent $1, 2, \ldots, 15$, then they are universal. This is the 15-theorem of Conway and Schneeberger with its beautiful proof by Bhargava \cite{Bh}. Here and for the rest of the paper we restrict ourselves to \emph{classical} quadratic forms, i.e.\ homogeneous quadratic polynomials with integral coefficients whose off-diagonal coefficients are even. 

Universal quadratic forms have also been investigated over number fields. Chan, Kim and Raghavan \cite{CKR} determined all totally positive universal ternary quadratic forms over $\Bbb{Q}(\sqrt{2})$, $\Bbb{Q}(\sqrt{3})$ and $\Bbb{Q}(\sqrt{5})$  (there are 4, 2, and 5, respectively, up to equivalence) and showed in addition that no other real quadratic number field admits totally positive universal ternary quadratic forms. (A form is called totally positive if it is positive and the form with conjugate coefficients is also positive.) 
The proof uses, among other things,  a theorem of Siegel \cite{Si} which states that in no totally real field other than $\Bbb{Q}$ and $\Bbb{Q}(\sqrt{5})$, every totally positive integer is a sum of \emph{any} number of squares. This is a first indication that more complicated fields might admit fewer universal quadratic forms. Kim \cite{Ki}  showed that for squarefree $D \geq 38446$, the field $\Bbb{Q}(\sqrt{D})$ admits no   diagonal septenary universal form. 

The aim of this paper is a proof that there exist infinitely many fields  that do not admit universal forms of \emph{arbitrary} length, diagonal or not. 

\begin{theorem}\label{thm1} Given any positive integer $M$,  there exist infinitely many real quadratic fields that do not admit classical universal quadratic forms in $M$ variables. 
\end{theorem}
 
Our proof constructs a very sparse sequence of such fields, about $e^{-cM}\sqrt{X}$ dis\-cri\-mi\-nants $D \leq X$ for some constant $c > 1$,  but it seems likely that such fields occur very frequently. Roughly speaking, if $K = \Bbb{Q}(\sqrt{D})$ has small class number, it contains many   integers of small norm, and we will see that  this often forces a universal form to have many variables. On the other hand, Kim \cite{Ki2} showed that there are infinitely many real quadratic fields that admit universal quadratic forms in 8 variables; these fields are all of the form $\Bbb{Q}(\sqrt{n^2+1})$ and have in particular very large class  number. 

This indicates already the difficulty of the proof of Theorem \ref{thm1}: although we expect that real quadratic fields often have small class number, it is one of the old unsolved problems in number theory to make any substantial progress in this direction. Our proof needs to work around the difficulty that we know very little about real quadratic fields with small class number. 
In the final section we make some explicit calculations for the field $K= \Bbb{Q}(\sqrt{73})$, which has class number 1. \\

\textbf{Acknowledgements.} We would very much like to thank the referee, not only for a careful reading of the manuscript, but also for very concrete and useful suggestions how to streamline the presentation and the argument. 

\section{Generalities} 

Throughout the paper we will use the following notation. Let $K=\mathbb Q(\sqrt D)$ with $D>0$ squarefree. If $D \equiv 2, 3$ (mod 4), then $\mathcal{O}_K = \Bbb{Z}[\sqrt{D}]$. If $D \equiv 1$ (mod 4), then $\mathcal{O}_K = \Bbb{Z}[\omega]$ where $\omega = (1 + \sqrt{D})/2$. In the former case we let $\delta := 2 \sqrt{D}$, in the latter $\delta :=  \sqrt{D}$.  
We write $a \succ b$ to mean that $a - b$ is totally positive, and we denote by $\mathcal O_K^+$ the set of totally positive integers. For $a\in K$ we denote its conjugate by $a^\prime$. The norm of $a$ is denoted by $Na = aa'$. 

In order to avoid repetition, by a   \emph{form} we mean a totally positive quadratic form with coefficients in $\mathcal{O}_K$ whose off-diagonal coefficients are even (i.e.\ divisible by 2).  The restriction to classical forms is quite typical in questions of universality (see e.g.\ \cite{Bh, CKR}); they are
convenient to work with since they can be  represented in the form as $x^{\top} A
x$ where $A$ is a symmetric matrix with integral coefficients. \\ 

We start with some simple lemmas.
\begin{lemma}\label{l1}
Let $a, b\in K$ be totally positive. Then $\sqrt{N(a+b)}\geq \sqrt{N(a)}+\sqrt{N(b)}$.
\end{lemma}
  
 \begin{proof} We have $$N(a+b)= N(a)+N(b)+ab\pr+a\pr b \geq N(a)+N(b)+2\sqrt{aa\pr bb\pr} =  (\sqrt{N(a)}+\sqrt{N(b)})^2.$$ 
 \end{proof}

 \begin{lemma}\label{l2}
Let $a\in\mathcal O_K^+$ be   such that $a\not\in\mathbb Z$ and $a>a\pr$. 
Then $a>\delta$.
\end{lemma}

\begin{proof}
Let $a=x+y\sqrt D$ and assume first that $x, y\in\mathbb Z$. Since $a\not\in\mathbb Z$, we have $y\neq 0$. Since $a>a\pr$ is totally positive, it follows that $x>0$ and $y>0$. Since $a\pr>0$, we have $x>y\sqrt D\geq \sqrt D$. Thus $a=x+y\sqrt D>2\sqrt D = \delta$.

If $x, y$ can be half-integers, we show the result in the same way.
\end{proof}

\begin{lemma}\label{p3}
Let $\alpha\in\mathcal O_K^+$. If $N\alpha \leq \delta$ and $n\nmid \alpha$ for all $2\leq n\in\mathbb Z$, then $\alpha$ is not a sum of two totally positive integers.
\end{lemma}

\begin{proof}
Assume that $\alpha=a+b$ with $a$, $b$ totally positive. If $ab\pr\not\in\mathbb Z$, then without loss of generality assume that $ab\pr>a\pr b$. It follows from  Lemma \ref{l2} that
$ab\pr>\delta$. This is a contradiction, since then $N(\alpha)=N(a)+N(b)+ab\pr+a\pr b>ab\pr>\delta$. Hence $ab\pr=a\pr b\in\mathbb N$, so that $q:=\frac{ab\pr}{N(a)} \in \Bbb{Q}$, say $q = \frac{u}{v}$ with coprime $u, v \in \Bbb{N}$. 
 Since $a, b\in\mathcal O_K$ and $b=\frac uv a$, we have that $v\mid a$. 
Then $a=vc$ and $b=uc$ for some $c\in\mathcal O_K$ and $\alpha=(u+v)c$, a contradiction.
\end{proof}

We can now make more precise the final remarks from the introduction: if  $K$ has sufficiently many elements of small norm satisfying some conditions, then it does not admit universal forms with few variables.

\begin{prop}\label{t4}
Assume that there exist $1=a_1, a_2, \dots, a_M\in\mathcal O_K^+$ such that for all $1 \leq i \not = j \leq M$ we have that 
\begin{enumerate}
\item\label{1}  $Na_i \leq \delta$,
\item\label{2} $n\nmid a_i $ for  $2\leq n\in\mathbb Z$, 
\item\label{3} $a_i$ and $a_j$ are not in the same square class, i.e.\ there exists no $x \in \mathcal{O}_K$ such that $a_i = a_j x^2$, and 
\item\label{4} If $a_ia_j \succ c^2$ for   $c \in \mathcal{O}_K$, then $c= 0$. 
\end{enumerate}
Then there are no universal classical totally positive $(M-1)$-ary quadratic forms over $\mathcal O_K$. 

Condition \eqref{4} is implied by 
\begin{enumerate} \setcounter{enumi}{4}
\item\label{5} $N(a_i)\leq\delta^{1/2}$ and $n \nmid a_ia_j$ for $2\leq n\in\mathbb Z$.
\end{enumerate}
\end{prop}
 
\begin{proof}  That \eqref{4} is implied by \eqref{5} follows from Lemma \ref{p3} applied to $a_ia_j$. 

The proof  of the proposition uses the language of lattices and is inspired by the escalation technique of Bhargava \cite{Bh}. Let $L_i=\langle a_1, \dots, a_i\rangle$ be the lattice associated to the diagonal form $\sum_{j\leq i} a_jx_j^2$. Let $L$ be a lattice representing $a_1, \ldots, a_M$. 
In particular it must have a vector of length $a_1 = 1$. Since $a_2$ is not a square, $L_1$ does not represent $a_2$, so $L$ must contain $$\left(\begin{matrix} a_1 & c\\ c & a_2\end{matrix}\right)$$
for some $c \in \mathcal{O}_K$. Since $L$ is totally positive, we have  $a_1a_2-c^2 \succ 0$,  so that  $c = 0$, and so $L$ contains   $L_2$. Having shown by induction that $L$ must contain $L_{i-1}$, we proceed similarly:  each of the  $a_{i}$ satisfies the assumptions of Lemma  \ref{p3},  so it is not a sum of totally positive elements.  Hence if $L_{i-1}$ represents it, we must have $a_i=a_jx_j^2$ for some $j \leq i-1$ which is not possible by assumption. 
Thus $L_{i-1}$ does not represent $a_{i}$, so  we need a linearly independent vector of length $a_i$. Hence $L$ contains $$\left(\begin{matrix} L_{i-1}& c\\ c & a_{i}\end{matrix}\right)$$ for some $c \in \mathcal{O}_K^{i-1}$. Considering $2$-by-$2$ subdeterminants, we see as above that each entry of $c$ must be 0.  It follows that $L$ contains $L_i$. 
In particular, we conclude that there cannot be a universal lattice of dimension $M-1$. 
\end{proof}

As mentioned in the introduction, for any given $M$, fields with the properties of Proposition \ref{t4} should exist in abundance, yet their existence is not easy to prove. 

\begin{prop} Assume  the Riemann hypothesis for $L(s, \chi_D)$, and that 
$K = \Bbb{Q}(\sqrt{D})$ has  narrow class number one.  Then for each $\varepsilon > 0$ there exists a  constant $C(\varepsilon)$ such that 
 a (classical) totally positive universal quadratic form needs at least $C(\varepsilon) D^{1/4 - \varepsilon}$ variables. 
\end{prop} 


\begin{proof} Let $I_n$ be the set of ideals of norm $n$, and let $r$ be the arithmetic function defined by $$r(n) =|I_n| =  \sum_{d \mid n} \chi_D(d).$$
Clearly $r(n)$ is bounded by the divisor function which itself is bounded by $c(\varepsilon)n^{\varepsilon}$ for every $\varepsilon > 0$. If $n$ is squarefree and $\mathfrak{a}   = \prod_{j} \mathfrak{p}_j \in I_n$ with distinct prime ideals $\mathfrak{p}_j$, then every ideal in $I_n$ is of the form $  \prod_{j} \tilde{\mathfrak{p}}_j $ where $\tilde{\mathfrak{p}}_j \in \{\mathfrak{p}_j, \mathfrak{p}'_j\}$. For each pair $ \{\mathfrak{p}, \mathfrak{p}'\}$ fix one of the two prime ideals, and for each squarefree $n$ for which $I_n \not= \emptyset$ fix a totally positive generator $\alpha(n)$ of the unique element of $I_n$ that is the product of our selected prime ideals. 
Then the set of all $\alpha(n)$ with $n \leq \delta^{1/2}$ satisfies the assumptions of Proposition \ref{t4}. If $\mu$ denotes the M\"obius function, then their cardinality is
$$\sum_{\substack{n\leq \delta^{1/2}\\ r(n) \not= 0}} \mu^2(n) \geq c(\varepsilon) D^{-\varepsilon} \sum_{n \leq \delta^{1/2}} \mu^2(n) r(n)  \geq c(\varepsilon) D^{-\varepsilon} \sum_{ n\leq \delta^{1/2} } \mu^2(n) r(n) \Bigl( 1- \frac{n}{\delta^{1/2}}\Bigr).$$
By Mellin inversion, the sum on the right hand side equals
$$\int_{2-i\infty}^{2+\infty} \sum_{n } \frac{r(n)\mu^2(n)}{n^s} \frac{\delta^{s/2}}{s(s+1)} \frac{ds}{2\pi i} = \int_{2-i\infty}^{2+\infty}  \zeta(s) L(s, \chi_D)  H(s)  \frac{ \delta^{s/2}}{s(s+1)} \frac{ds}{2\pi i},$$
where the Euler product 
 $$H(s) = \prod_{\chi_D(p) = 1}\left(1 - \frac{3}{p^{2s}} + \frac{2}{p^{3s}}\right) \prod_{\chi_D(p) \not= 1} \left(1 - \frac{1}{p^{2s}}\right) $$
is absolutely convergent and uniformly bounded from above and below in $\Re s \geq 1/2 + \varepsilon$. We evaluate the integral by shifting the contour to line $\Re s =  2/3$, say,  and picking up the residue of the pole at $s=1$. The Riemann hypothesis implies the Lindel\"of hypothesis \cite[p.\ 116]{IK}, so that $L(2/3+ it , \chi_D) \ll_{\varepsilon}  ((1+|t|)D)^{\varepsilon}$. Hence 
$$\int_{2-i\infty}^{2+i\infty}  \zeta(s) L(s, \chi_D)  H(s)  \frac{\delta^{s/2}}{s(s+1)} \frac{ds}{2\pi i} = \frac{1}{2}L(1, \chi_D)H(1)  \delta^{1/2}+ O(\delta^{1/3} D^{\varepsilon}) \gg \frac{D^{1/4}}{\log\log D},$$
since the Riemann hypothesis also implies $L(1, \chi_D) \gg 1/\log \log D$ \cite{Li}. This completes the proof. 
 \end{proof}

\emph{Remark:} Replacing ideals $\mathfrak{a} \in I_n$ by $\mathfrak{a}^h$ if the class number of $K$ is $h \geq 1$, the same proof shows that   a (classical) totally positive universal quadratic form needs in general at least $C(\varepsilon) D^{1/(4h) - \varepsilon}$ variables.

 

\section{Squarefree values of quadratic and linear polynomials}

The aim of this section is a proof of the following essentially classical result.

\begin{lemma}\label{squarefree}
Let $f(x) = ax^2 + bx + c$ be an integral quadratic polynomial with discriminant $\Delta = b^2 - 4ac \not = 0$.  For $j = 1, \ldots, m$ let  $g_j(x) = k_j x + r_j$ be    linear integral polynomials.  Assume that each of $f, g_1, \ldots, g_m$ takes at least one squarefree value. 
Then for a positive  proportion of natural numbers $n$, the values $f(n), g_1(n), \ldots, g_m(n)$ are simultaneously squarefree. 
More precisely, the asymptotic formula
$$\mathcal{S} := \sum_{ n \leq X } \mu^2(f(n))  \prod_{j=1}^m \mu^2(g_j(n)) = CX + O\left(X^{1 - \frac{2}{3(m+1)} + \varepsilon}\right)$$
holds for a constant $C > 0$. Here $C$ and the $O$-constant depend on $f$ and the $g_j$. 
\end{lemma}

\emph{Remark:} By replacing $f(x)$ with $f(4x)$, we can also guarantee  $f(n) \equiv c \pmod 4$. 
 
\begin{proof}  
  We have 
$$\mathcal{S} =  
\sum_{ n \leq X }\mu^2(f(n)) \prod_{j=1}^m \sum_{d_j^2 \mid g_j(n)} \mu(d_j). $$
Let $Y \geq 2$ be a parameter to be chosen later. We would like to approximate $\mathcal{S}$ by
$$\mathcal{S}_0 := \sum_{ n \leq X  }\mu^2(f(n)) \prod_{j=1}^m \sum_{\substack{ d_j^2 \mid g_j(n)\\ d_j \leq Y}} \mu(d_j).$$
To estimate the error, let $1 \leq \mu \leq m$ and define
$$\mathcal{S}_{\mu} := \sum_{ n \leq X }\mu^2(f(n)) \prod_{j=1}^{\mu-1} \sum_{\substack{d_j^2 \mid g_j(n) \\ d_j   \leq Y}}\mu(d_j) \sum_{\substack{d_{\mu}^2 \mid g_{\mu}(n) \\ d_{\mu} > Y} } \mu(d_{\mu}) \prod_{j=\mu+1}^m \sum_{d_j^2 \mid g_j(n)}  \mu(d_j).$$
Then 
\begin{displaymath}
\begin{split}  |\mathcal{S}_{\mu} | 
\leq &  \sum_{n \leq X}  \prod_{j=1}^{\mu-1} \sum_{\substack{d_j^2 \mid g_j(n) \\ d_j   \leq Y}}\mu^2(d_j)  \sum_{\substack{d_{\mu}^2 \mid g_{\mu}(n) \\ d_{\mu} > Y} } \mu^2(d_{\mu})   \leq    \sum_{ \substack{d_1, \ldots, d_{\mu-1}   \leq Y\\ Y < d_{\mu} \leq \sqrt{|k_{\mu}|X + |r_{\mu}|}}} \mu^2(d_1) \cdots \mu^2(d_{\mu}) \sum_{\substack{n \leq X \\ d_j^2 \mid g_j(n) } } 1 \\
\leq  & \sum_{ \substack{d_1, \ldots, d_{\mu-1}  \leq Y\\ Y < d_{\mu} \leq \sqrt{|k_{\mu}|X + |r_{\mu}|}}}  \left(\frac{\mu^2(d_1) \cdots \mu^2(d_{\mu}) }{[d_1^2, \ldots, d_{\mu}^2]} X + O(1)\right)\\
  \leq  & X \sum_{  d_1, \ldots, d_{\mu-1}} \sum_{d_{\mu} > Y}    \frac{\mu^2(d_1) \cdots \mu^2(d_{\mu}) }{[d_1^2, \ldots, d_{\mu}^2]} + O\left(Y^{\mu-1}X^{1/2}\right). 
\end{split}
\end{displaymath}
To estimate the multiple sum, we apply Rankin's trick and choose $0 < s < 1$. Then the sum is bounded by
\begin{equation}\label{trick}
 \frac{1}{Y^s} \sum_{  d_1, \ldots, d_{\mu}} \frac{\mu^2(d_1) \cdots \mu^2(d_{\mu}) d_{\mu}^s}{[d_1^2, \ldots, d_{\mu}^2]} \leq \frac{1}{Y^s} \prod_p \left(1 + \frac{2^{\mu} p^s}{p^2}\right) \ll Y^{-s}.
 \end{equation}
With $s = 1 - \varepsilon$ we conclude $\mathcal{S}_{\mu} \ll XY^{\varepsilon - 1} + Y^{\mu-1} X^{1/2}$, and hence
\begin{displaymath}
\begin{split}
\mathcal{S}& = \mathcal{S}_0 + \sum_{\mu=1}^{m} \mathcal{S}_{\mu} = \mathcal{S}_0 + O\left(XY^{\varepsilon - 1} + Y^{m-1} X^{1/2}\right).
\end{split}
\end{displaymath}
We proceed to manipulate $\mathcal{S}_0$. Let $Z > Y$ be another parameter. We have
$$\mathcal{S}_0 = \sum_{d_1, \ldots, d_m \leq D} \mu(d_1) \cdots \mu(d_m) \sum_{d_0} \mu(d_0) \sum_{ \substack{n \leq X  \\ d_0^2 \mid f(n) \\ d_j^2 \mid g_j(n)}} 1 = \mathcal{S}^{(1)} + \mathcal{S}^{(2)},$$
say, where $\mathcal{S}^{(1)}$ is the contribution of $d_0 \leq Z$ and $\mathcal{S}^{(2)}$ is the contribution $d_0 > Z$. 
 We first bound $\mathcal{S}^{(2)}$. We write $$d_0^2 k = f(n) = \frac{(2an + b)^2 - \Delta}{4a}.$$ For given $a, k, \Delta \not= 0$, the equation $4ak d_0^2 - (2an+b)^2 = -\Delta \not= 0$ is of Pellian type and has at most $O(\log X)$ solutions $(d_0, n)$ with $n \leq X$ (with an absolute implied constant). We obtain
\begin{displaymath}
\begin{split}
|\mathcal{S}^{(2)}|& \leq Y^m  \sum_{n \leq X} \sum_{\substack{d_0^2k = f(n) \\ d_0 > Z}} 1  \ll Y^m \sum_{k \leq  (|a|X^2 + |b|X + |c|)/Z^2} \log X \ll Y^m X^{2+\varepsilon}Z^{-2}. 
\end{split}
\end{displaymath}
In order to evaluate $\mathcal{S}^{(1)}$, we write $\textbf{d} = (d_0, d_1, \ldots, d_m)$ for squarefree $d_0, \ldots, d_m$ and define
$$\rho(\textbf{d}) := |\{n \, (\text{mod } [d_0^2, \ldots, d_m^2]) \mid f(n) \equiv 0 \, (\text{mod } d_0^2), \, g_j(n) \equiv 0 \, (\text{mod } d_j^2)\}|.$$
For notational simplicity we also write $\mu(\textbf{d}) = \mu(d_0) \cdots \mu(d_m)$. 
By the Chinese remainder theorem, $\rho$ is multiplicative in each variable and $\rho(p, d_1, \ldots, d_m) \leq 2$ for all primes $p \nmid \Delta$, hence $\rho(\textbf{d}) \ll \tau(d_0)$, where $\tau$ is the divisor function.  Moreover, 
\begin{equation}\label{nonzero}
  \rho(\textbf{d}) < [d_0^2, \ldots, d_m^2]
 \end{equation} 
for all $\textbf{d} \not= (1, \ldots, 1)$,    for otherwise at least one of $f, g_1, \ldots, g_m$ would have a fixed square divisor.   Now
\begin{displaymath}
\begin{split}
\mathcal{S}^{(1)} & =   \sum_{d_1, \ldots, d_m \leq Y} \sum_{d_0 \leq Z} \mu(\textbf{d}) \left(\frac{\rho(\textbf{d})}{[d_0^2, \ldots, d_m^2]} X + O(\rho(\textbf{d}))\right)\\
&=  \sum_{d_1, \ldots, d_m \leq Y} \sum_{d_0 \leq Z} \frac{ \mu(\textbf{d}) \rho(\textbf{d})}{[d_0^2, \ldots, d_m^2]} X  + O\left(Y^m Z\log Z\right).
\end{split}
\end{displaymath}
The same argument as in \eqref{trick} implies that we may complete the multiple sum at the cost of an error $O(X Y^{\varepsilon - 1})$, and the resulting Euler product is absolutely convergent with value $C \not = 0$   by \eqref{nonzero}. Combining everything, we have shown
\begin{displaymath}
\begin{split}
\mathcal{S} &= CX + O\left( Y^m Z\log Z + Y^mX^{2+\varepsilon} Z^{-2} + XY^{\varepsilon - 1} + Y^{m-1} X^{1/2}\right),
\end{split}
\end{displaymath}
and the lemma follows upon choosing $Z = X^{2/3}$, $Y = X^{\frac{2}{3(m+1)}}.$ 
\end{proof}

\section{Continued fractions and elements of small norm}

First we collect some useful results on continued fractions. 
Let $\gamma=[a_0, a_1, \dots]$ be an infinite continued fraction of a real number $\gamma > 0$, let and $p_i/q_i=[a_0, \dots, a_i]$ be its $i$th approximation ($a_i, p_i, q_i\in\mathbb N$). Then it is easy to see and well-known that 
$p_{i+1}=a_{i+1}p_i+p_{i-1}$ and  $q_{i+1}=a_{i+1}q_i+q_{i-1}$ 
and 
$$\left\lvert \frac{p_i}{q_i} - \gamma \right\rvert<
\left\lvert \frac{p_i}{q_i} - \frac{p_{i+1}}{q_{i+1}} \right\rvert=
\frac {1}{q_iq_{i+1}}<
\frac {1}{a_{i+1}q_{i}^2}.$$
 Assume now that $\gamma=\sqrt D$ (with squarefree $D$) and let $$\alpha_i=p_i+q_i\sqrt D, \quad N_i=N(\alpha_i)=p_i^2-Dq_i^2.$$ Then
 $|p_i-q_i\sqrt D|<\frac 1{a_{i+1}q_i}\leq \frac 1{q_{i}}$ and so
 \begin{equation}\label{sizeN}
 |N_i|=(p_i+q_i\sqrt D)|p_i-q_i\sqrt D|<
\Bigl(2q_i\sqrt D+\frac 1{q_i}\Bigr)\frac {1}{a_{i+1}q_{i}}=
\frac {2\sqrt D}{a_{i+1}}+\frac {1}{a_{i+1}q_{i}^2}.
\end{equation}
Since $(p_i, q_i) = 1$, we see that $\alpha_i$ is not divisible by a rational integer $\geq 2$. \\ 

\emph{Remark:} We see that if $a_{i+1}$ is not too small compared to $\sqrt D$, then $N_i$ has a small norm. There are several explicit   examples of such continued fractions. For instance, 
take $b, n, k\in\mathbb N$ and let 
$$D=D(b, n, k)=(b(1+2bn)^k+n )^2+2(1+2bn)^k.$$
Then for each $j<k$, we have that $b_j=2b(1+2bn)^{k-j}$ appears as a coefficient in the continued fraction for $\sqrt D$ (see \cite[Section 3]{Ma}).  We have $\lfloor\sqrt D\rfloor=b(1+2bn)^k+n$, and so 
$$\frac {2\sqrt D}{b_j}<\frac{b(1+2bn)^k+n+1}{b(1+2bn)^{k-j}}<
(1+2bn)^{j}+1< D^\frac{j}{2k}+1.$$
Although this produces many elements of small norm, it is not so easy to choose $b, n, k$ so that $D$ is (almost) squarefree, and it is also not trivial to verify the other conditions of Proposition \ref{t4}. Therefore we proceed slightly differently in the following. \\



Let us now consider periodic continued fractions of the form 
\begin{equation}\label{gamma}
\gamma=[k; \overline{u, \dots, u, 2k}]
\end{equation}
with $\ell$ elements $u$ in the period. Friesen \cite{Fr} gave, for general periodic, symmetric continued fractions,  certain necessary and sufficient parity conditions that ensure that there are infinitely many $k$ such that  $\gamma=\sqrt D$ with squarefree $D$. We shall need more explicit information than this (in particular on the convergents $p_i/q_i$, see Proposition \ref{construct alpha} below), so let's compute the special case \eqref{gamma} in complete detail.  Let
\begin{equation}\label{rhoc}
\rho_{\pm} := \frac{1}{2}(u \pm \sqrt{u^2 + 4}), \quad c_{\pm} := \frac{\pm 1 \pm (k + \sqrt{D})\rho_\pm}{\sqrt{4+u^2}}, \quad c'_{\pm} := \frac{\pm 1 \pm (k - \sqrt{D})\rho_\pm}{\sqrt{4+u^2}}.
\end{equation}
Since $q_{-1}=0$, $q_0=1$, $q_1=u$, we have
$$q_i = \frac{\rho_+^{i+1} - \rho_-^{i+1}}{\rho_+ - \rho_-}$$
for $i \leq \ell$, so that $q_i$ is a function of $u$. Then we can evaluate  $p_i$ by 
  $p_i=kq_i+q_{i-1}$ for $i \leq \ell$.

Clearly, 
  $$\rho_+ > u, \quad \rho_+\rho_- = -1, \quad -\frac{1}{u} < \rho_- < 0, \quad c_+ > c_-. $$
We start with the following essentially well-known lemma. 
\begin{lemma}\label{q-lemma}  Let $\gamma$ be as in \eqref{gamma} and keep the notation developed so far. \\
{\rm a)}  For $0 \leq j < i \leq \ell$ we have $$q_iq_{j-1}-q_{i-1}q_j=(-1)^{j+1}q_{i-j-1} \quad \text{and} \quad 
q_iq_{i-2}-q_{i-1}^2=(-1)^i.$$
{\rm b)} 
For $0 \leq i \leq \ell$, the sequences  $p_i$, $q_i$,  $\alpha_i$ and $\alpha_i'$ are linear combinations of $\rho_+^i$ and $\rho_-^i$, and we have explicitly
  $$\alpha_i = c_+\rho^i_+ + c_- \rho^i_- \quad \text{and} \quad  \alpha'_i = c'_+\rho^i_+ + c'_- \rho^i_-.$$
 {\rm c)} For $0 \leq i \leq \ell$ and $u$ even, we have $q_i\equiv 0\pmod 2$ when 
 $i\equiv 1\pmod 2$. 
   \end{lemma}

\begin{proof} \ \\a) We have
\begin{displaymath}
\begin{split}
& q_iq_{j-1}-q_{i-1}q_j=(uq_{i-1}+q_{i-2})q_{j-1}-q_{i-1}(uq_{j-1}+q_{j-2})\\
= &
(-1)(q_{i-1}q_{j-2}-q_{i-2}q_{j-1})=\dots=
(-1)^{j}(q_{i-j}q_{-1}-q_{i-j-1}q_{0})=(-1)^{j+1}q_{i-j-1},
\end{split}
\end{displaymath}
because $q_0=1$ and $q_{-1}=0$. We obtain the second identity by taking $j=i-1$.\\
b) This follows in a well-known fashion from the recurrence $\alpha_{i+1} = u \alpha_i + \alpha_{i-1}$ for $i \leq \ell - 1$. \\
c) This follows directly from the recurrence relation. 
\end{proof}

The following lemma contains some technical estimates for the quantities defined in \eqref{rhoc} for future reference. 
\begin{lemma}\label{technical}\ \\
{\rm a)} If $k \geq u$, then  $c_- > 0$.\\
{\rm b)}  If $u \geq 2$, then $\rho_-^2 < |c_-'| < 1.$\\
{\rm c)}  If $u \geq 2$, then $|c'_+| < 2 \rho_+^{-\ell}=2|\rho_-|^\ell$.\\ 
{\rm d)} If $n \leq (\ell - 4)/2$ and $u \geq 2$, then $|c_+' \rho_+^n| < |c_-' \rho_-^n|/2$. 
\end{lemma}

\begin{proof} \ \\
a) We observe that $k + \sqrt{D} > 2k \geq 2u > \rho_+$, so that $-1 - (k + \sqrt{D}) \rho_- > 0$. \\
b)  We have $[(\sqrt{D} - k)^{-1}] =u$, so that 
$0 > k - \sqrt{D} >     - u^{-1}.$ 
This implies 
$$c'_- \sqrt{4+u^2} = -1-(k - \sqrt{D}) \rho_- < -1,$$ so that in particular $c'_- < 0$ and $$|c_-'| > \frac{1}{\sqrt{4+u^2}} > \frac{1}{2u} \geq \frac{1}{u^2} > \rho^2_-.$$ For the other inequality we have 
$$|c_-'| \sqrt{4 + u^2} = 1 + (k- \sqrt{D}) \rho_- < 1 - \frac{2}{u} \rho_- < 1 + \frac{2}{u^2} < \sqrt{4 + u^2}.$$
c)   From b) we have $|c_-' \rho_-^\ell| < 1$. Since $|\alpha_\ell'| < 1$, we must have $|c_+' \rho_+^\ell| < 2$.\\ 
d) Since $\rho_+ > u \geq 2$, we conclude from part c) that $$|c_+' \rho_+^n| < 2 \rho_+^{n-\ell} < \frac{1}{2}\rho_+^{n-\ell+2}.$$ On the other hand, from part b) we obtain $$|c_-' \rho_-^n| > |\rho_-^{n+2} | = \rho_+^{-2-n}.$$
Combining the last two displays proves the claim for $n < (\ell - 4)/2$.
\end{proof}

\begin{prop}\label{t-express}
Let $\gamma $ be as in \eqref{gamma}.  
Then $\gamma=\sqrt D$ for some $D\in\mathbb N$ if and only if $q_\ell\mid kp_\ell+p_{\ell-1}$.  When $u$ and $q_{\ell}$ have the same parity, this condition is satisfied when $2k= q_\ell t+u$ for $t \in \Bbb{N}$ if $u$ is even, and $t \in \Bbb{N}$ odd if $u$ is odd. 
In this case, 
\begin{equation}\label{D}
D = D(t)=k^2+tq_{\ell-1}+1=t^2\frac{q_\ell^2}4+t\frac{uq_\ell+2q_{\ell-1}}{2}+\frac{u^2}{4}+1.
\end{equation}
Moreover, if $\alpha_i=p_i+q_i\sqrt D$ for $i \leq \ell$, then 
\begin{equation}\label{ni}
 N_i=N(\alpha_i)=(-1)^{i+1}(tq_iq_{\ell-i-1}+1)
 \end{equation}
  is a linear polynomial in $t$.
\end{prop}

\begin{proof}
It is  a well-known property of continued fractions that
$$\frac{(k+\gamma)p_\ell+p_{\ell-1}}{(k+\gamma)q_\ell+q_{\ell-1}}=\gamma.$$
Using  $p_\ell=kq_\ell+q_{\ell-1}$ it then simplifies to 
$\gamma^2q_\ell=kp_\ell+p_{\ell-1}$.
We see that $\gamma=\sqrt D$ if and only if $q_\ell \mid kp_\ell+p_{\ell-1}=k^2q_\ell+2kq_{\ell-1}+q_{\ell-2}$; in other words, $q_\ell=uq_{\ell-1}+q_{\ell-2} \mid 2kq_{\ell-1}+q_{\ell-2}$, we obtain that we can take $2k=q_\ell t +u$, and \eqref{D} follows.  

Since $p_i=kq_i+q_{i-1}$, we get
\begin{displaymath}
\begin{split}
N_i&=p_i^2-Dq_i^2=(kq_i+q_{i-1})^2-(k^2+tq_{\ell-1}+1)q_{i}^2\\
&=
k^2q_i^2+2kq_iq_{i-1}+q_{i-1}^2-k^2q_i^2-tq_{\ell-1}q_{i}^2-q_{i}^2\\
&=
(q_\ell t+u)q_iq_{i-1}+q_{i-1}^2 -tq_{\ell-1}q_{i}^2-q_{i}^2\\
&=
q_it(q_\ell q_{i-1}-q_{\ell-1}q_i)+q_{i-1}(q_{i-1}+uq_i)-q_i^2\\
&=q_it(-1)^{i+1}q_{\ell-i-1}+q_{i-1}q_{i+1}-q_i^2=
(-1)^{i+1}(tq_iq_{\ell-i-1}+1),
\end{split}
\end{displaymath}
where in the last two equalities we have used the two identities from Lemma \ref{q-lemma}a) and the recurrence $q_{i+1}=uq_i+q_{i-1}$ for $i \leq \ell-1$. 
When $i=\ell$, then $q_{\ell+1}$ is not given by this recurrence, but we still have $q_{\ell-1}(q_{\ell-1}+uq_\ell)-q_\ell^2=(-1)^{\ell+1}$.
\end{proof}

We are now ready to state and prove the first key ingredient for the proof of Theorem \ref{thm1}. 

\begin{prop}\label{construct alpha}
For every  $u \in \Bbb{N}$ with   $u \equiv 2\pmod {4}$ and 
$\frac{1}{4}u^2+1$ squarefree and every   odd $\ell \in \Bbb{N}$, there exist  infinitely many squarefree $D\equiv 2\pmod 4$ with $\sqrt D=[k; \overline{u, \dots, u, 2k}]$ with $\ell$ elements $u$ in the period, such that for $1 \leq i <  j \leq (\ell-1)/2$ and $i \equiv j \equiv 1 \pmod 2$, we have 
\begin{enumerate}
\item $N(\alpha_i)$ is squarefree and not $\pm1$,
\item $n\nmid\alpha_i$ for $2 \leq n\in\mathbb N$,
\item $N(\alpha_i)<2\frac{\sqrt D}{u}+\frac 1u\frac{\sqrt D}{u}$ and $\alpha_i \succ 0$. 
\end{enumerate}
\end{prop}

\begin{proof}  Since $u$ is even and $\ell$ is odd, it follows from Lemma \ref{q-lemma}c) that $q_{\ell}$ is even, so that Proposition \ref{t-express} is applicable. Recall from \eqref{ni} that
\begin{equation}\label{recall}
N_i = N_i(t)   = k_i t + 1, \quad k_i = q_iq_{\ell-i-1}
\end{equation}
for $i$ odd. In particular, $\alpha_i$ is not a unit for $i \leq \ell-1$, which is the second part of condition (1).  It follows from the recurrence relation similarly as in Lemma \ref{q-lemma}b) that
$$k_i = \alpha^2 \rho_+^{\ell-1} + \beta^2 \rho_-^{\ell-1} + \alpha \beta\left( (-1)^{\ell-i-1} \rho_+^{2i+1 - \ell} + (-1)^i \rho_{+}^{\ell - 2i - 1}\right)$$
for certain real numbers $\alpha, \beta$. 
The last parenthesis of the previous display is strictly decreasing in absolute value for $i \leq (\ell-1)/2$, so that $k_i \not= k_j$ for $1 \leq i <  j\leq (\ell-1)/2$ and $i \equiv j \equiv 1\pmod 2$. 

Conditions (2) and (3)  are automatic; notice that \eqref{recall} implies   $\alpha_i \succ 0$   if $i$ is odd. 
The first part of condition (1)  follows directly from Lemma \ref{squarefree} applied to the quadratic polynomial in \eqref{D} and the linear polynomials in \eqref{recall}. Notice that the constant terms of $D(t)$ and $N_i(t)$ are squarefree, so all considered polynomials take squarefree values. 
\end{proof}

The $\alpha_i$ constructed in the previous proposition satisfy properties \eqref{1} -- \eqref{3} from Proposition \ref{t4}, but a comparison of \eqref{D} and \eqref{ni} shows that $N\alpha_i \asymp \sqrt{D}$, so they fail to satisfy condition \eqref{5} from Proposition \ref{t4}. Hence we must find another way to ensure \eqref{4}, and the rest of this section is devoted to this task. \\
 
From now on we  assume that $\sqrt D=[k; \overline{u, \dots, u, 2k}]$ for a squarefree $D\equiv 2, 3\pmod 4$ with $k > u \geq 5$.  For $i \leq \ell-1$ we know that $\alpha_i=p_i+q_i\sqrt D$ is totally positive if and only if $i$ is odd, and we know from \eqref{sizeN} that  $$|N_i|=|N(\alpha_i)|<\frac {2\sqrt D}{u}+\frac {1}{uq_{i}^2}<\frac{1}{2}\sqrt D.$$

\begin{lemma}\label{noh}
Take odd $i<j\leq(\ell-4)/2$. There is no $h\geq 0$ such that
$\alpha_{i}\alpha_{j}\succ\alpha_h^2$.
\end{lemma}

\begin{proof}
Assume that $\alpha_{i}\alpha_{j}\succ\alpha_h^2$.
First note that the sequence $\alpha_n$ is strictly increasing, and so $h<2j+1$. Hence the explicit formulae from Lemma \ref{q-lemma}b) hold for $\alpha_i$, $\alpha_j$ and $\alpha_h$. 
From Lemma \ref{technical}a) we have that  $c_+>c_->0$. Now, for $h \geq 1$ we have  
\begin{displaymath}
\begin{split}
c_+^2\rho_+^{2h-\frac 12}&<c_+^2(\rho_+^{2h}-2)<c_+^2\rho_+^{2h}+2c_+c_-(-1)^h<\alpha_h^2\\
&<
\alpha_i\alpha_j=(c_+\rho_+^i+c_-\rho_-^i)(c_+\rho_+^j+c_-\rho_-^j)<
c_+^2\rho_+^{i+j}
\end{split}
\end{displaymath}
(in the last inequality we have used that $\rho_-$ is negative and $i, j$ are odd).
Hence $2h-\frac{1}{2}<i+j$, and so $2h\leq i+j$. With slightly more precise estimates we can even exclude the case $2h = i+j$. In this case, the inequality $\alpha_i\alpha_j > \alpha_h^2$ is after some simplification equivalent to 
$$c_-c_+(\rho_+^i\rho_-^j + \rho_-^i\rho_+^j) > 2 c_-c_+(-1)^h.$$
Since  $c_-$ and     $c_+$ are positive and $i, j$ are odd, the left hand side is negative, so that $h$ must be odd, and we obtain $2 > \rho_+^{i-j} +  \rho_+^{j-i} $ which is a contradiction. We conclude that $2h < i+j$. 

 On the other hand, since $i, j \leq (\ell - 4)/2$ we have  from Lemma \ref{technical}(d) that 
$$\alpha_i' \alpha_j' < (|c_+' \rho_+^i| + |c_-'\rho_-^i|)(|c_+' \rho_+^j| + |c_-'\rho_-^j|) < \Bigl(\frac{3}{2} \Bigr)^2 |{c'_-}^2| |\rho_-|^{i+j}.$$
If $2h \leq i+j \leq \ell -4$, then  from Lemma \ref{technical}b) and c) we have 
 \begin{displaymath}
 \begin{split}
   (\alpha_h')^2 & = (c'_-\rho_-^h)^2 + 2 c_-'c_+'\rho_+^h\rho_-^h + (c_+' \rho_+^h)^2 > (c'_-\rho_-^h)^2 - 2 |c_-'c_+'| \\
   &> (c'_-\rho_-^h)^2 - 4 |c_-'\rho_-^{\ell}|  > (c'_-\rho_-^h)^2 (1 - 4|\rho_-|^{\ell - 2 - 2h}) > (c'_-\rho_-^h)^2 (1 - 4u^{-2}) = \frac{21}{25} (c'_-\rho_-^h)^2.
   \end{split}
 \end{displaymath}
Now our assumption $\alpha_i'\alpha_j' \geq (\alpha_h')^2$ implies $ |\rho_-|^{i+j} > \frac{1}{3} |\rho_-|^{2h} >  |\rho_-|^{2h+1}$, so that $i+j < 2h+1$, and hence $i+j \leq 2h$. This contradicts our previous conclusion $2h < i+j$, so that the original assumption $\alpha_{i}\alpha_{j}\succ\alpha_h^2$ was wrong. 
 \end{proof}

\begin{prop}\label{prop11}
Take odd $i<j\leq(\ell-4)/2$. 
Assume that there is $\mu\in\mathcal O_K$ such that 
$\alpha_i\alpha_j\succ\mu^2$. Then $\mu=0$.
\end{prop}

\begin{proof} If $\mu \in \Bbb{Z}$, then $1>\alpha_i\pr\alpha_j\pr>\mu^2$, which is possible only for $\mu=0$. From now on assume that 
  $\mu\not\in\mathbb Z$. We will show that this  implies $\mu=\alpha_h$ which is a contradiction by the previous lemma.  

Let $\alpha_i\alpha_j=\mu^2+\nu$ with $\nu \succ 0$. Since also $\mu^2 \succ 0$, we have 
$\frac 14 D>N(\alpha_i\alpha_j)>N(\mu)^2$, and so $|N(\mu)|< \frac{1}{2} \sqrt D $. Let $\mu=x+y\sqrt D\not\in\mathbb Z$. Clearly $x$ and $y$ must have the same sign (otherwise $\alpha_i'\alpha_j' > (\mu')^2$ cannot be satisfied), and without loss of generality assume that 
$x, y>0$. We distinguish two cases. 

\emph{Case 1:} $\mu\pr<0$. Then $y^2D-x^2=|N(\mu)|<\frac{1}{2} \sqrt D $, and so $y^2D-\frac{1}{2}\sqrt D -x^2<0$, hence 
$\sqrt D$ lies between the roots of $2y^2 T^2-{T}-2x^2$.
Thus 
$$\sqrt D<\frac{1+\sqrt{1+16x^2y^2}}{4y^2}<
\frac{1}{4y^2}+\frac{\sqrt{1+8xy+16x^2y^2}}{4y^2}=
\frac xy+\frac 1{2y^2}.$$
By assumption $\frac xy<\sqrt D$, and so we see that $|\frac xy-\sqrt D|<\frac 1{2y^2}$, 
but as is well-known \cite[Theorem 184]{HW}, this implies that   $\mu=\alpha_h$ for some $h$.

\emph{Case 2:} $\mu\pr>0$. Then $x^2-y^2D=|N(\mu)|<\frac{1}{2} \sqrt D $, and so $2y^2D+\sqrt D-2x^2>0$. Since $\sqrt D>0$, we must have that $\sqrt D$ is greater than the positive root of the polynomial, i.e., 
$$\sqrt D>\frac{-1+\sqrt{1+16x^2y^2}}{4y^2}>
\frac{-1}{4y^2}+\frac{\sqrt{1-8xy+16x^2y^2}}{4y^2}=
\frac xy-\frac 1{2y^2}.$$
Since $\sqrt D<\frac xy$, we again obtain $|\frac xy-\sqrt D|<\frac 1{2y^2}$, and so 
$\mu=\alpha_h$ for some $h$.
\end{proof}

The  proof of Theorem \ref{thm1} is now a direct consequence of Propositions \ref{t4}, \ref{construct alpha} and \ref{prop11}.
 
\section{An example}

We conclude the paper with an explicit example for the field $\Bbb{Q}(\sqrt{73})$.  We will see that in such specific situation one can get a little further than the general results of the previous sections. We suppress some of the fairly straightforward computations. 
 

Since $73\equiv 1\pmod 4$, the ring of integers $\mathcal O_K=\mathbb Z[\omega]$, where $\omega=\frac {1+\sqrt{73}}{2}\approx 4.772$. 
We know from Lemma \ref{p3} that if $\alpha\in\mathcal O_K^+$ has norm $\leq 8<\sqrt {73}$ and is not divisible by any $n\in\mathbb Z$, $n>1$, then $\alpha$ is not the sum of two totally positive integers. We shall use the following elements of $\mathcal O_K$:
\begin{itemize}
\item $\varepsilon=943+250\omega\approx 2136$ is the fundamental unit, $N(\varepsilon)=-1$
\item $\rho=4+\omega\approx 8.772$ and $\rho\pr=5-\omega\approx 0.228$ are totally positive elements of norm 2
\item $\sigma=83+22\omega\approx 187.984$ and $\sigma\pr=105-22\omega\approx 0.016$ are 
totally positive elements of norm 3
\end{itemize}

We shall often use the following variant of Lemma \ref{p3}:

\begin{lemma}\label{lemma73}
Let $\gamma\in\{\rho, \rho\pr, \sigma, \sigma\pr, \rho\sigma, \rho\sigma\pr, \rho\pr\sigma, \rho\pr\sigma\pr, \allowbreak
\rho^2\sigma, \rho\prd\sigma\pr, \allowbreak
\rho\sigma\prd, \rho\pr\sigma^2, 
\allowbreak
2\rho, 2\rho\pr, \allowbreak 2\sigma, \allowbreak 2\sigma\pr, \allowbreak 
2\rho\sigma, \allowbreak 2\rho\pr\sigma\pr, \allowbreak
3\rho, 3\rho\pr\}$,
$\alpha\in\mathcal O_K$ and $\beta\in\mathcal O_K^+$. If $\gamma=\alpha^2+\beta$, then $\alpha=0$.
\end{lemma}

\begin{proof}
If $\gamma\in\{\rho, \rho\pr, \sigma, \sigma\pr, \rho\sigma, \rho\sigma\pr, \rho\pr\sigma, \rho\pr\sigma\pr\}$, then the claim follows from Lemma \ref{p3}, as all these elements have norm $\leq 6<\sqrt{73}$.

Assume that $\alpha\neq 0$. Without loss of generality we can assume that $\alpha>0$.
By Lemma \ref{l1} we know that $\sqrt{N(\gamma)}\geq |N(\alpha)|+\sqrt{N(\beta)}$. Since 
$\gamma$ has norm at most 24, it follows that $|N(\alpha)|\leq \sqrt{24}-1\approx 3.9$. Hence $N(\alpha)=\pm 1, \pm 2, \pm 3$. 

We have  $\alpha^2<\gamma\leq \rho^2\sigma\approx 14465$, and so 
$\alpha<121$. Similarly $\alpha\prd<121$.
Since $\varepsilon$ is fairly big, we can check that the only possibilities are
$\alpha=1, \rho, \rho\pr, \varepsilon\sigma\pr, \varepsilon\pr\sigma$.
One then just checks that in each of these cases, $\gamma-\alpha^2$ is not totally positive.
\end{proof}
 
Note that three cases which do not satisfy Lemma \ref{lemma73} are 
$2\rho\pr\sigma\succ\rho^2$, $\rho\sigma^2\succ (\varepsilon\rho\pr)^2$ and $\rho^2\sigma\pr\succ 1^2$.

\

Let us now start constructing a universal quadratic form $L$:
$L$ has to represent 1, and so $L=L_1\oplus L_1\pr$ for $L_1=\langle 1\rangle$.

Since $\rho$ is not a square, $L_1$ does not represent $\rho$. Let 
$$L_2=\left( \begin{array}{cc}
1 & a \\
a & \rho \\
\end{array} \right).$$

We have $1\cdot \rho\succ a^2$, and so $a=0$ by Lemma \ref{lemma73}. Hence we obtain that $L=L_2\oplus L_2\pr$ with $L_2=\langle 1, \rho\rangle$.

Similarly $L_2$ does not represent $\sigma$.
Let 
$$L_3=
\left( \begin{array}{ccc}
1 & 0 & a \\
0 & \rho & b \\
a & b & \sigma \end{array} \right)$$
be an escalation of $L_2$. Then $\sigma\succ a^2$ and $\rho\sigma\succ b^2$, and so by Lemma \ref{lemma73} we see that $a=b=0$ and $L_3=\langle 1, \rho, \sigma\rangle$.

Now $L_3$ does not represent $\rho\pr$. Again we can show that all the off-diagonal elements after the escalation by $\rho\pr$ are zero, except for the one in the same row as $\rho$, so that
$$L_4=\langle 1, \sigma\rangle
\oplus\left( \begin{array}{cc}
\rho & a \\
a & \rho\pr \\
\end{array} \right),$$ 
where $2=\rho\rho\pr\succ a^2$. Thus $a=0, \pm 1$. However, in the case $a=-1$, the form is equivalent to the form with $a=1$. Hence it suffices to consider $a=0, 1$.

Next $L_4$ does not represent $\sigma\pr$: if it did, we would need to have $a=1$ and $\rho x^2+2xy+\rho\pr y^2=\sigma\pr$, which implies
$(\rho x+y)^2+y^2=\rho\sigma\pr$, a contradiction with Lemma \ref{lemma73}.
Let $L_5$ be the corresponding escalation.
We can again show that some off-diagonal elements are zero; we are left with 
$$L_5=\langle 1\rangle
\oplus\left( \begin{array}{cc}
\rho & a \\
a & \rho\pr \\
\end{array} \right)
\oplus\left( \begin{array}{cc}
\sigma & b  \\
b & \sigma\pr \end{array} \right),$$
where $b=0, 1$.

The only way how $L_5$ could represent 2
is if $a=1$ and $\rho x^2+2xy+\rho\pr y^2 \in \{1, 2\}$, 
or if $b=1$ and $\sigma x^2+2xy+\sigma\pr y^2\in \{1, 2\}$. 
These two cases imply
$(\rho x+y)^2+y^2 \in \{\rho, 2\rho\}$ and 
$(\sigma x+ y)^2+2y^2\in \{\sigma, 2\sigma\}$, respectively. But all of these conditions contradict Lemma \ref{lemma73}.

The escalation of $L_5$ by 2 is 
$$L_6^0=
\left( \begin{array}{cc}
\rho & a \\
a & \rho\pr \\
\end{array} \right)
\oplus\left( \begin{array}{cc}
\sigma & b  \\
b & \sigma\pr \end{array} \right)
\oplus\left( \begin{array}{cc}
1 &  c \\
c &  2 \end{array} \right).$$
We have $c^2\prec 2$, and so $c=0, 1$. If $c=1$,  the last matrix is equivalent to the diagonal form 
$$\left( \begin{array}{cc}
1 & 0 \\
0 & 1 \end{array} \right).$$

Hence we conclude that $L_6^0$ is equivalent  to the form 
$$L_6=\langle 1, m\rangle
\oplus\left( \begin{array}{cc}
\rho & a \\
a & \rho\pr \\
\end{array} \right)
\oplus\left( \begin{array}{cc}
\sigma & b  \\
b & \sigma\pr \end{array} \right),$$
where $m=1, 2$.

In the same manner as before we verify that $\rho\pr\sigma$ is not represented by $L_6$. The corresponding escalation is
$$L_7=\langle 1\rangle
\oplus\left( \begin{array}{cccc}
\rho & a & 0 & 0\\
a & \rho\pr & c & 0\\
0 & c & \rho\pr\sigma & d \\
0 & 0 & d & m
\end{array} \right)
\oplus\left( \begin{array}{cc}
\sigma & b  \\
b & \sigma\pr \end{array} \right),$$
where $c=0, 1$; $d=0$ if $m=1$ and $d=0, \rho$ if $m=2$.

Finally let us show that $L_7$ does not represent $\rho\sigma$. 
For this we distinguish two cases according to the value of $c$:

$c=0$: Then $$L_7=\langle 1\rangle
\oplus\left( \begin{array}{cc}
\rho & a \\
a & \rho\pr \\
\end{array} \right)
\oplus\left( \begin{array}{cc}
\sigma & b  \\
b & \sigma\pr \end{array} \right)
\oplus\left( \begin{array}{cc}
\rho\pr\sigma & d \\
d & m \end{array} \right).$$
As before, if $\rho\sigma$ were represented by this form, 
at least one of the equations 
$x^2+(x+\rho\pr y)^2=\rho\rho\pr\sigma=2\sigma$, 
$2x^2+(x+\sigma\pr y)^2=\rho\sigma\sigma\pr=3\rho$, and
$\rho x^2+(\rho x+2y)^2=2\rho\sigma$ (note here that $2\rho\pr\sigma-\rho^2=\rho$) would have a solution, 
but this is not possible by Lemma \ref{lemma73}.

$c=1$: Let us first show that in this case, $a=d=0$. Assume that $a=1$.
But then 
$$\left( \begin{array}{ccc}
\rho & 1 & 0 \\
1 & \rho\pr & 1 \\
0 & 1 & \rho\pr\sigma
\end{array} \right)$$
has determinant $\rho\pr\sigma-\rho$, which is not totally positive, a contradiction.

Assume now that $d=\rho$, which is possible only when $m=2$. Then
$$\left( \begin{array}{ccc}
 \rho\pr & 1 & 0\\
 1 & \rho\pr\sigma & \rho \\
 0 & \rho & 2
\end{array} \right)$$
has determinant $2\rho\prd\sigma-2\rho-2=0$, so it is again not totally positive.

Hence we see that in the case $c=1$, 
$$L_7=\langle 1, m, \rho\rangle
\oplus\left( \begin{array}{cc}
\rho\pr & 1 \\
1 & \rho\pr\sigma \\
\end{array} \right)
\oplus\left( \begin{array}{cc}
\sigma & b  \\
b & \sigma\pr \end{array} \right).$$
As before we verify that this form does not represent $\rho\sigma$.

\

We could probably continue a little longer by escalating $L_7$ by $\rho\sigma$, but the situation is becoming messy, so let's stop here by concluding that we have shown:

\begin{prop}
A (classical) universal  totally positive   quadratic forms over 
$\mathbb Z[\frac {1+\sqrt{73}}{2}]$ must have at least 8 variables. 
\end{prop}
  
Since the elements $1, 2, \rho, \rho', \sigma, \sigma', \rho \sigma, \rho \sigma', \rho' \sigma, \rho' \sigma'$ cover 10 different square classes and (with the exception of $2=1+1$) are not sums of totally positive elements, we can also conclude that a universal \emph{diagonal} form over $\mathbb Z[\frac {1+\sqrt{73}}{2}]$ must have at least 10 variables.

\end{document}